\documentclass[11pt,english,a4paper]{amsart}
\pdfoutput=1

\usepackage[utf8]{inputenc}
\usepackage[T1]{fontenc}
\usepackage{lmodern}
\usepackage{babel}
\usepackage{amsmath,amssymb,amsthm}
\usepackage[marginratio={1:1,1:1},totalheight=600pt]{geometry}
\usepackage{microtype}
\usepackage{hyperref}
\usepackage{url}
\usepackage{tikz-cd}
\usepackage{colortbl}
\usepackage[normalem]{ulem}

\theoremstyle{plain}
\newtheorem{theorem}{Theorem}[section]

\newtheorem{proposition}[theorem]{Proposition}
\newtheorem{corollary}[theorem]{Corollary}

\theoremstyle{definition}
\newtheorem{definition}[theorem]{Definition}
\newtheorem{example}[theorem]{Example}
\newtheorem{remark}[theorem]{Remark}
\newtheorem{question}[theorem]{Question}

\theoremstyle{remark}

\newcommand{\Z}{\mathbb{Z}}

\newcommand{\F}{\mathbb{F}}

\title[Reversible cellular automata]{New classes of reversible cellular automata}

\author[Haugland]{Jan Kristian Haugland}
\address{Norwegian National Security Authority (NSM)\\Norway}
\email{admin@neutreeko.net}

\author[Omland]{Tron Omland}
\address{Norwegian National Security Authority (NSM) \and Department of Mathematics, University of Oslo\\Norway}
\email{tron.omland@gmail.com}

\date{January 30, 2026}

\begin{document}
	
\begin{abstract}
Cellular automata find applications in cryptography, error correction, theoretical computer science, and the simulation of physical and biological processes. We say that a Boolean function is a proper lifting if its induced cellular automaton is reversible. This paper introduces new classes of proper liftings, constructed through compositions of landscape functions, and explores their cryptographic potential by computing indicators relevant to differential and linear cryptanalysis. We combine theoretical and computational methods to obtain a complete classification of proper liftings for diameters up to~$6$. Several open problems and conjectures are posed along the way.

\end{abstract}
	
\maketitle

\section*{Introduction}

Cellular automata have long been significant in computational and applied mathematics, with applications in cryptography and scientific modeling. A particularly intriguing class is that of reversible cellular automata, which can simulate reversible processes in physics and biology. They are also relevant to theoretical computer science, where their bijective nature provides insights into reversible computation, fault-tolerant systems, and efficient algorithmic design. The study of cellular automata is closely related to shift-invariant functions, which are essential to symmetric cryptography, especially lightweight cryptography, in particular the design of substitution boxes (S-boxes) for block ciphers and hash functions. Shift-invariant functions and their associated cellular automata have previously been studied by several authors, see e.g., \cite{Mariot-24} for a survey.

Recall that a function is called shift-invariant if it commutes with the (right) shift operator. Any shift-invariant function on $n$-bits arises from a local rule, that is, a Boolean function on $k$-bits for some $k\leq n$. Conversely, every Boolean function $f \colon \F_2^k \to \F_2$ induces a shift-invariant function $F\colon\F_2^n\to\F_2^n$ for every $n\geq k$. If for every such $n$ the induced map $F$ is bijective, then $f$ is called a proper lifting, and the corresponding cellular automaton is reversible.

A reversible binary cellular automaton thus induces, by restriction to finite windows, a family of bijections on $\F_2^n$ that are shift-invariant and have a very explicit, local description. Constructions based on landscapes are particularly attractive since composing such local rules leads to reversible maps built from simple, uniform local updates, and hence to structured families of permutations that can be evaluated efficiently. This structure also constrains classical cryptographic criteria, so we view these families primarily (i) as components under strong implementation constraints and (ii) as objects for classification and further theoretical development. Nevertheless, nontrivial behavior is still possible in small diameter. Among the proper liftings of diameter~$6$ we identify a composition that simultaneously attains the best values we observe for both differential uniformity and nonlinearity. We define primitive landscapes as a special class of conserved (involutive) landscape rules. Using them as building blocks, we obtain a complete classification of proper liftings for $k\leq 6$, which further motivates the conjecture that reversible cellular automata may all arise as compositions of primitive landscapes.

In Section~\ref{prelims}, we formalize a subclass of involutive landscape rules, called primitive landscape functions, and describe a new way to produce proper liftings through composition of such functions, generalizing existing constructions coming from sets of landscapes, see e.g., \cite{JDA-thesis}. In Section~\ref{diameter-6}, we study in detail the case $k=6$, and obtain a complete classification. Finally, in Section~\ref{other-constructions}, we present a few other constructions of families of proper liftings.

By combining computational and theoretical results, this work constructs new families of reversible cellular automata and provides a complete classification up to diameter~$6$, contributing to the understanding of these functions.

\subsection*{Acknowledgements}
The authors thank Emmanuel Jeandel for helpful correspondence and the
anonymous referees for valuable comments on the manuscript.
The second author was partially supported by the Research Council of
Norway (RCN), project no.~345433.

\section{Proper liftings and reversible cellular automata}\label{prelims}

A Boolean function $f\colon\F_2^k\to\F_2$ induces a function $F\colon\F_2^n\to\F_2^n$ for every $n\geq k$, by
\[
F(x_1,x_2,\dotsc,x_n)=\big(f(x_1,x_2,\dotsc,x_k),f(x_2,x_3,\dotsc,x_{k+1}), \dotsc,f(x_n,x_1,\dotsc,x_{k-1})\big),
\]
which is shift-invariant (or rotation-symmetric), that is, commutes with the right shift operator.
In fact, every shift-invariant function arises this way. If $f$ has diameter~$k$, i.e., depends on both $x_1$ and $x_k$, and $F$ is bijective, we say that $f$ is a $(k,n)$-lifting. If $f$ is a $(k,n)$-lifting for every $n\geq k$, then $f$ is called a \emph{proper lifting} (see~\cite{HO-1}).

A Boolean function $f$ of diameter~$k$ also induces a function $F\colon\F_2^\Z\to\F_2^\Z$, defined by
\[
F(x)_i = f(x_i,x_{i+1},\dotsc,x_{i+k-1})
\]
for all $i\in\Z$, which is called a cellular automaton.
It is known that $F$ is bijective if and only if $f$ is a proper lifting, see \cite[Theorem~7]{kari-survey}, and in this case $F$ is called a \emph{reversible cellular automaton}.

\begin{remark}
If $F\colon\F_2^n\to\F_2^n$ is a shift-invariant function and the smallest and largest indices of the variables that $F(x)_1$ depend upon are $i$ and $j$, respectively, then $F$ has diameter $k=j-i+1$. If $\sigma$ denotes the right shift, then the function $F'=F\circ \sigma^{1-i}$ is also shift-invariant, and $F'(x)_1$ depends on both $x_1$ and $x_k$, but no $x_\ell$ for $\ell>k$.
\end{remark}

The following is well-known (we note that it can be deduced from Proposition~\ref{permutive-proper} below):
\begin{proposition}
An affine Boolean function is a proper lifting if and only if it has diameter~$1$.
\end{proposition}

One way to generate examples in larger diameters is by stretching known proper liftings:

\begin{proposition}\label{expand}
Let $f\colon\F_2^k\to\F_2$ be a Boolean function, $s$ any integer $\geq 2$, and define the function $f_s\colon\F_2^{(k-1)s+1}\to\F_2$ by
\[
f_s(x_1,\dotsc,x_{(k-1)s+1})=f(x_1, x_{s+1}, \dotsc, x_{(k-1)s+1}).
\]
If $f$ is a $(k, n)$-lifting for all $n \geq k$, then $f_s$ is a $((k-1)s+1, n)$-lifting for all $n \geq (k-1)s+1$.
\end{proposition}

\begin{proof}
Take any integer $n \geq (k-1)s+1$. Since $\frac{n}{\gcd(n, s)}$ is an integer $\geq \frac{n}{s} > k-1$, $f$ is a $(k, \frac{n}{\gcd(n, s)})$-lifting. Let $F$ be the corresponding bijection. Moreover, $$(x_1, \dotsc, x_n) \longleftrightarrow \{(x_1, x_{s+1}, \dotsc), (x_2, x_{s+2}, \dotsc), \dotsc, (x_{\gcd(n, s)}, x_{s+\gcd(n, s)}, \dotsc)\}$$ is a bijection. The result follows by applying $F$ to the tuples on the right-hand side.
\end{proof}

\begin{example}
Let $f\colon\F_2^4\to\F_2$ be given by $f(x)=x_2\oplus x_1(x_3\oplus 1)x_4$, which is known to be a proper lifting, first observed in \cite{patt}. Then $f_3\colon\F_2^{10}\to\F_2$ is given by $f_3(x)=f(x_1,x_4,x_7,x_{10})$ and is also a proper lifting.
\end{example}

When we say that $f(x_1, \dotsc, x_k) \oplus x_j$ does not depend on $x_j$, we mean that there exists a function $g$ in $k-1$ variables such that $f(x)=x_j\oplus g(x_{\neq j})$.

\begin{proposition}\label{shift-product}
Let $f\colon\F_2^k\to\F_2$ be a Boolean function of diameter~$k$, and suppose that for some $1\leq j\leq k$, the function $f(x_1, \dotsc, x_k) \oplus x_j$ does not depend on $x_j$. Furthermore, suppose that for each $t$, $1-j \leq t \leq k-j$ such that $f(x_1, \dotsc, x_k) \oplus x_j$ depends on $x_{j+t}$, we have $(f(x_1, \dotsc, x_k) \oplus x_j) (f(x_{1+t}, \dotsc, x_{k+t}) \oplus x_{j+t}) = 0$. Then $f$ is a $(k, n)$-lifting for each $n \geq k$.
\end{proposition}

\begin{proof}
We prove the equivalent statement that $F\colon\F_2^n\to\F_2^n$ given by $$F(x_1, \dotsc, x_n)_i = f(x_{i-j+1}, \dotsc, x_{i-j+k})$$ is a bijection, where the indices are considered $\operatorname{mod} n$. More specifically, we prove that $F(F(x)) = x$. To this end, we verify that for any $i$, $F(x)_i = x_i$ if and only if $F(F(x))_i = F(x)_i$. Thus, either $F(x)_i = x_i$ and $F(F(x))_i = F(x)_i$, or $F(x)_i \neq x_i$ and $F(F(x))_i \neq F(x)_i$. In both cases, $F(F(x))_i = x_i$.

Let $T= \{t \mid f(x_1, \dotsc, x_k) \oplus x_j \text{ depends on } x_{j+t}\}$, and let $x \in \F_2^n$ and $i \in \{0, \dotsc, n-1\}$.

(a) Suppose $F(x)_i \neq x_i$. This is equivalent to $f(x_{i-j+1}, \dotsc, x_{i-j+k}) \oplus x_i = 1$. By the requirement of the proposition, it follows that $F(x)_{i+t} = x_{i+t}$ for any $t \in T$. Since $F(x)_i$ only depends on the $x_{i+t}$, we conclude that $F(F(x))_i \neq F(x)_i$.

(b) Next, we apply the contrapositive of (a) to $x_{i+t}$, $F(x)_{i+t}$ and $F(F(x))_{i+t}$. I.e., if $F(F(x))_{i+t} = F(x)_{i+t}$ for any $t \in T$, then it follows that $F(x)_{i+t} = x_{i+t}$ for any $t \in T$.

(c) Suppose $F(F(x))_i \neq F(x)_i$. Just like in (a), it follows that $F(F(x))_{i+t} = F(x)_{i+t}$ for any $t \in T$. By (b), it follows that $F(x)_{i+t} = x_{i+t}$ for any $t \in T$. And like in (a), we conclude that $F(x)_i \neq x_i$.

Combining (a) and (c) shows that $F(x)_i = x_i$ if and only if $F(F(x))_i = F(x)_i$.
\end{proof}

\begin{example}
Again, let $f\colon\F_2^4\to\F_2$ be the function given by $f(x)=x_2\oplus x_1(x_3\oplus 1)x_4$. Then $j=2$, and for $t=-1$ we get that $x_1(x_3\oplus 1)x_4 \cdot x_4(x_2\oplus 1)x_3 = 0$, and similarly for $t=1,2$, which gives a way of showing that $f$ is indeed a proper lifting.
\end{example}

Cellular automata on a bit sequence can be specified by one or more possible shapes of some nearby bits, called \emph{landscapes} (cf.~\cite{toffoli-margolus}). A landscape is a string of symbols $\epsilon_{1-s} \dotsc \epsilon_{-1} \star \epsilon_1 \dotsc \epsilon_{k-s}$, where $\epsilon_i \in \{0, 1, -\}$ for each $i$, $\epsilon_{1-s} \in \{0, 1\}$ and $\epsilon_{k-s} \in \{0, 1\}$. A bit $x_0$ is inverted if there is a landscape for which $x_i = \epsilon_i$ for each $i$ for which $\epsilon_i \in \{0, 1\}$. In the case with one single landscape, this corresponds to a Boolean function $f$ of diameter~$k$ given by $$f(x) = x_s \oplus \prod_{i|\epsilon_i \in \{0, 1\}} (x_{s+i} \oplus \epsilon_i \oplus 1).$$

\begin{remark}
Usually the indexing of the variables of a Boolean function starts from 1, but for landscapes it is standard notation that the variable corresponding to $\star$ is indexed 0. When dealing with landscapes and compositions of Boolean functions (cf.\ Definition~\ref{composition}) it is often convenient to shift the indexing (i.e., apply $x_i \leftarrow x_{i + 1}$ or $x_i \leftarrow x_{i - 1}$  an appropriate number of times), and we have already seen this in the proof of Proposition~\ref{shift-product}.
\end{remark}

The following corollary is a generalization of Theorem 7.4 in \cite{OS-iacr}.

\begin{corollary}\label{primitive-landscapes}

Suppose a cellular automaton is given by a landscape $\epsilon_{1-s} \dotsc \epsilon_{-1} \star \epsilon_1 \dotsc \epsilon_{k-s}$, and that for each $d_1$ such that the $d_1$th symbol is either 0 or 1, there exists $d_2$ such that the $d_2$th and the $(d_1+d_2)$th symbol of the landscape consist of 0 and 1 (in some order). Then the cellular automaton is reversible, and the corresponding Boolean function is a proper lifting.
\end{corollary}

\begin{proof}

Let $f$ be the corresponding Boolean function, and let $d_1 \in T$, as in the proof of Proposition~\ref{shift-product}. By the assumption of the corollary, there exists $d_2$ such that $f(x_1, \dotsc, x_k) \oplus x_s$ is a multiple of both $(x_{s+d_2} \oplus \epsilon_{d_2} \oplus 1)$ and $(x_{s+d_1+d_2} \oplus \epsilon_{d_1+d_2} \oplus 1)$. For the former factor, we can substitute the $x$'s by adding $d_1$ to each index and find that $f(x_{1+d_1}, \dotsc, x_{k+d_1}) \oplus x_{s+d_1}$ is a multiple of $(x_{s+d_1+d_2} \oplus \epsilon_{d_2} \oplus 1)$. Furthermore, the factors $(x_{s+d_1+d_2} \oplus \epsilon_{d_2} \oplus 1)$ and $(x_{s+d_1+d_2} \oplus \epsilon_{d_1+d_2} \oplus 1)$ have distinct values, so that one of them is 0. Thus, $(f(x_1, \dotsc, x_k) \oplus x_s) (f(x_{1+d_1}, \dotsc, x_{k+d_1}) \oplus x_{s+d_1}) = 0$ is satisfied and the conclusion follows from Proposition~\ref{shift-product}.
\end{proof}

\begin{definition}
A function that satisfies the hypothesis of Corollary~\ref{primitive-landscapes}, often given in the landscape notation described above, will be called a \emph{primitive landscape} function. 
\end{definition}

It follows from the proof of Proposition~\ref{shift-product} that primitive landscape functions satisfy $F^2=I$. Rules with this involutive behavior are called \emph{conserved landscapes} in \cite{toffoli-margolus} and are also referred to as \emph{marker cellular automata} in the literature. Intuitively, as explained in \cite{MPJL-evolutionary}, reversibility of a landscape rule $f(x)=x_s\oplus g(x_{\neq s})$ is easiest to maintain when $g$ has low Hamming weight, which tends to limit attainable nonlinearity in this constrained class.

\begin{remark}\label{elementary}
Note that two strings of symbols yield \emph{elementary equivalent} functions, as defined in \cite[Definition~2.10]{HO-1}, if one is obtained from the other by reading it backwards and/or by replacing each $0$ and $1$ by its complement.

For $k=4$, there is, up to elementary equivalence, only one proper lifting, described by any of the four following primitive landscapes (where the first one is the Patt function):
\[
(0 \star 10), \quad
(1 \star 01), \quad
(01 \star 0), \quad
(10 \star 1).
\]

A brute-force search for primitive landscapes gave the results in the following table. The Python program that was used is available at \cite{landscapes-github}. It essentially counts elementary equivalence classes and extrapolates the total number. Thus, it assumes that the first symbol is 0, that $2 \leq s \leq \frac{k+1}{2}$, and compares each string with its mirror image in the case $s = \frac{k+1}{2}$. Otherwise, it tries out all possible strings of symbols from $\{0, 1, -\}$. The relative density is the number of primitive landscapes divided by the total number of landscapes, which is given by $4 (k-2)3^{k-3}$.
\medskip
\begin{center}
\begin{tabular}{|c|c|c|c|}\hline
$k$ & No.\ of primitive landscapes & Relative density & No.\ of equivalence classes \\ \hline
4 & 4 & 0.1667 & 1 \\ \hline
5 & 14 & 0.1296 & 4 \\ \hline
6 & 72 & 0.1667 & 18 \\ \hline
7 & 288 & 0.1778 & 73 \\ \hline
8 & 1160 & 0.1989 & 290 \\ \hline
9 & 4376 & 0.2144 & 1100 \\ \hline
10 & 16776 & 0.2397 & 4194 \\ \hline
11 & 60646 & 0.2568 & 15176 \\ \hline
12 & 219344 & 0.2786 & 54836 \\ \hline
13 & 775930 & 0.2986 & 194047 \\ \hline
14 & 2724072 & 0.3204 & 681018 \\ \hline
15 & 9394778 & 0.3400 & 2348878 \\ \hline
16 & 32291160 & 0.3617 & 8072790 \\ \hline
17 & 109326972 & 0.3810 & 27332464 \\ \hline
18 & 368586536 & 0.4014 & 92146634 \\ \hline
\end{tabular}
\end{center}
\end{remark}

\medskip

\begin{definition}\label{composition}
If $f$ and $g$ are two Boolean functions of diameter $k_f$ and $k_g$, respectively, inducing shift-invariant $F$ and $G$ for some $n\geq\max\{k_f,k_g\}$, then $G\circ F$ is shift-invariant and induced from the function $$(g\circ f)(x)=g\left( f(x_1, \dotsc, x_{k_f}), \dotsc, f(x_{k_g}, \dotsc, x_{(k_f + k_g - 1)}) \right)$$ with indices shifted so that $x_1$ is the first variable on which it depends. Thus, the diameter $k$ of $g \circ f$ satisfies $k \leq k_f + k_g - 1$.
\end{definition}

\begin{example}
Suppose $f(x) = x_2 \oplus x_1 (x_3 \oplus 1) x_4$ and $g(x) = x_2 \oplus x_1 (x_3 \oplus 1) (x_4 \oplus 1) x_5$, which are proper liftings, and let $x \in \F_2^8$. Then
\[
g\left( f(x_1, x_2, x_3, x_4), \dotsc, f(x_5, x_6, x_7, x_8) \right)
\]
is equal to
\[
h(x_1, \dotsc, x_8) = f(x_2, x_3, x_4, x_5) \oplus f(x_1, x_2, x_3, x_4) (f(x_3, x_4, x_5, x_6) \oplus 1) \dotsm
\]
Suppose we alter $x_1$. Then $h(x)$ can only change if the factor $f(x_1, x_2, x_3, x_4)$ on the right-hand side changes. This is only possible if $x_3 = 0$ and $x_4 = 1$. But in that case, the second factor $f(x_3, x_4, x_5, x_6) \oplus 1$ is equal to zero. Thus, $h(x)$ does not change, and is independent of $x_1$. Similarly, $h(x)$ is independent of $x_7$ and $x_8$, and is in fact equal to $x_3 \oplus x_2 (x_5 (x_4 \oplus 1) \oplus (x_5 \oplus 1) x_6 (x_3 \oplus x_4 \oplus 1))$. The shift $x_i \leftarrow x_{i+1}$ yields the diameter~$5$ function
\[
g \circ f = x_2 \oplus x_1 (x_4 (x_3 \oplus 1) \oplus (x_4 \oplus 1) x_5 (x_2 \oplus x_3 \oplus 1)).
\]
\end{example}

\begin{remark}
If $f$ and $g$ are elementary equivalent to $f'$ and $g'$, respectively, then $g\circ f$ is typically not elementary equivalent to $g'\circ f'$. Indeed, if $c$ and $r$ are the maps given by complementing and reflecting (see \cite[Definition~2.10]{HO-1}), then $(g\circ r^i\circ c^j)\circ f$ give potentially four inequivalent classes, one for each choice of $i,j\in\{0,1\}$.
\end{remark}

If $f$ is a $(k, n)$-lifting and $m \geq k$ is a divisor of $n$, then $f$ is a $(k, m)$-lifting (cf.~\cite[Proposition~6.1]{JDA-thesis}). It follows that if $f$ is a $(k, n)$-lifting for all sufficiently large $n$, then $f$ is a $(k, n)$-lifting for all $n \geq k$.

\begin{proposition}
Suppose $f$ and $g$ are proper liftings. Then $g \circ f$ is also a proper lifting.
\end{proposition}

\begin{proof}
Since both the induced shift-invariant functions $F$ and $G$ are reversible cellular automata, it follows that $G \circ F$ is a reversible cellular automaton, and that $g \circ f$ is a proper lifting.
\end{proof}

This allows us to construct new proper liftings as compositions of primitive landscape functions. The proper liftings of diameter~$\leq 5$ (cf.\ \cite[Appendix~C]{HO-1}) can be written as follows, where a composition of landscapes means the composition of the corresponding functions:

\[
\begin{array}{r c l}
x_2 \oplus x_1 (x_3 \oplus 1) x_4 &=& 1 \star 01\\
x_2 \oplus x_1 x_3 (x_4 \oplus 1) (x_5 \oplus 1) &=& 1 \star 100\\
x_2 \oplus x_1 (x_3 \oplus 1) (x_4 \oplus 1) x_5 &=& 1 \star 001\\
x_2 \oplus x_1 (x_4 (x_3 \oplus 1) \oplus (x_4 \oplus 1) x_5 (x_2 \oplus x_3 \oplus 1)) &=& (1 \star 001) \circ (1 \star 01)\\
x_3 \oplus x_1 x_2 (x_4 \oplus 1) x_5 &=& 11 \star 01\\
x_3 \oplus x_1 (x_2 \oplus 1) x_4 (x_5 \oplus 1) &=& 10 \star 10

\end{array}
\]

In \cite{JDA-thesis}, Daemen consider shift-invariant functions defined by multiple landscapes, given as a set. In this case, a bit flips if it is the origin of at least one landscape in the set. Every shift-invariant function defined by this set-wise construction can also be defined through the composition as given above. To illustrate how this works, one can check that
$0 \star 110 \vee 10 \star 10$ corresponds to $0 \star 110 \circ 10 \star 10$. Furthermore, $0 \star 10 \vee 0--\star10 \vee 0----\star10$ corresponds to $(0-1-1\star10) \circ (0-1\star10) \circ (0\star10)$.
On the other hand, for a function $f$ to be defined by a set of landscapes it is necessary that $f(x)\oplus x_s$ is independent of $x_s$ for some $s$, which is not the case for the fourth function in the above list, i.e., the one given by a composition.

\section{Proper liftings of diameter 6}\label{diameter-6}

The constructions from Section~\ref{prelims} yield large families of proper liftings arising from primitive landscapes and their compositions.
We have found $120$ elementary equivalence classes ($472$ functions with no constant term) by using the primitive landscape functions of diameter~$\leq 6$ as generators and closing under composition (and elementary equivalence). They are described in Appendix~\ref{appendix:landscapes} below.
To show that this list is actually complete, we compare it with an independent enumeration due to Jeandel~\cite{jeandel} of the $27{,}961{,}145$ elementary equivalence classes of non-permutive surjective cellular automata of neighborhood size~$6$.

\begin{proposition}\label{prop:diam6-complete}
There are exactly $120$ elementary equivalence classes of proper liftings of diameter~$6$.
In particular, every non-affine proper lifting of diameter~$6$ is a composition of primitive landscapes.
\end{proposition}

\begin{proof}
Among the elementary equivalence classes in Jeandel's list, every class outside the $120$ listed in Appendix~\ref{appendix:landscapes}
fails to be a $(6,n)$-lifting for some $n\leq 12$ (in particular, it is not a proper lifting).
Finally, we separately exhaust the permutive diameter~$6$ rules (cf.\ Proposition~\ref{permutive-proper} and the subsequent comment) and verify that no additional proper liftings occur there. Hence the $120$ classes listed in Appendix~\ref{appendix:landscapes} are all the proper liftings of diameter~$6$ up to elementary equivalence.
\end{proof}

Recall that the differential uniformity of a function $F\colon\F_2^n\to\F_2^n$ is given by
\[
\operatorname{DU}(F)=\max_{a,b\in\F_2^n,\, a\neq 0} \lvert\{ x\in\F_2^n : F(x \oplus a) \oplus F(x)=b \}\rvert.
\]
For a Boolean function $f\colon\F_2^k\to\F_2$, set
\[
\Delta(f)=\min_{m\geq k}\max_{n\geq m} \, 2^{-n} \operatorname{DU}(F).
\]
Below, for each pair of diameter~$k$ and degree, we list the unique function with the lowest value of $\max_{k\leq n\leq 12} 2^{-n} \operatorname{DU}(F)$, and to avoid fractions, the numbers given for each $n$ are for $2^{9-n}\cdot\operatorname{DU}(F)$:
\begingroup\small
\setlength{\arraycolsep}{4pt}
\[
\begin{array}{|c|c|c|c|c|c|c|c|c|c|c|c|} \hline
\text{Function} & k & \deg & n=4 & 5 & 6 & 7 & 8 & 9 & 10 & 11 & 12
\\ \hline
(0 \star 10) & 4 & 3 & 192 & 224 & 240 & 216 & 216 & 216 & 216 & 216 & 216
\\ \hline
(0 \star 110) \circ (0 \star 10) & 5 & 4 &  - & 128 & 144 & 144 & 136 & 132 & 132 & 132 & 132
\\ \hline
(0-\star100)\circ(0-\star110) & 6 & 3 & - & - & 192 & 224 & 224 & 216 & 240 & 216 & 216
\\ \hline
(00\star10)\circ(0\star110)\circ(0\star10) & 6 & 4 & - & - & 144 & 128 & 132 & 120 & 117 & 117 & 117
\\ \hline
(0\star10)\circ(0\star110)\circ(01\star00) & 6 & 5 &  - & - & 80 & 104 & 84 & 72 & 72 & 72 & 72
\\ \hline
\end{array}
\]
\endgroup

\begin{question}
For every diameter and degree, what is the lowest value of $\Delta(f)$?
\end{question}

\begin{question}
Does there exist a proper lifting of degree~$2$? We have checked that there are none of diameter~$\leq 7$, and that no composition of two proper liftings of degree~$\leq 6$ gives a degree~$2$ function. Moreover, by Proposition~\ref{shift-product}, every primitive landscape function has degree at least $3$.

Note that, by Proposition~\ref{expand}, the Patt function $(0 \star 10)$ can be expanded to proper liftings $(0 - \star - 1 - 0)$ for arbitrarily large lengths of $-$, that is, there exist proper liftings of degree~$3$ of arbitrarily large diameter.
\end{question}

The nonlinearity $\operatorname{nl}(h)$ of a Boolean function $h$ is the minimum Hamming distance between $h$ and the set of affine functions. For $F\colon\F_2^n\to\F_2^n$ we use the standard nonlinearity defined by
\[
\operatorname{NL}(F)=\min_{v\neq 0}\operatorname{nl}(v\cdot F).
\]
Note that $\operatorname{NL}(F)=0$ may occur even when $F$ is non-affine, and indeed this happens for some proper liftings.
\begin{samepage}
For $n=7,8$, the largest values of $\operatorname{NL}(F)$ attained by proper liftings in our diameter and degree ranges are achieved by the same examples that minimize $\operatorname{DU}(F)$ (though not uniquely for $k=6$ and $\deg=4$), and below we give the values of $2^{4-n}\cdot\operatorname{NL}(F)$:
\[
\begin{array}{|c|c|c|c|c|c|c|c|} \hline
\text{Function} & k & \deg & n=4 & 5 & 6 & 7 & 8
\\ \hline
(0 \star 10) & 4 & 3 & 2 & 2 & 2 & 2 & 2
\\ \hline
(0 \star 110) \circ (0 \star 10) & 5 & 4 &  - & 3 & 3 & 3 & 3 
\\ \hline
(0-\star100)\circ(0-\star110) & 6 & 3 & - & - & 2 & 2 & 2 
\\ \hline
(00\star10)\circ(0\star110)\circ(0\star10) & 6 & 4 & - & - & 0 & 3 & 3 
\\ \hline
(0\star10)\circ(0\star110)\circ(01\star00) & 6 & 5 &  - & - & 3 & 4 & 4 
\\ \hline
\end{array}
\]
\end{samepage}
In particular, of all proper liftings of diameter~$\leq 6$, the composition $(0\star10)\circ(0\star110)\circ(01\star00)$ simultaneously achieves the best values both for differential uniformity and nonlinearity.

\section{Other constructions}\label{other-constructions}

While primitive landscapes yield all proper liftings of diameter~$6$, and may plausibly do so in general, we include here a few alternative constructions. They serve as additional sources of examples, and we briefly discuss how they interact with composition and elementary equivalence.

\begin{proposition}
For integers $k$ and $j$ satisfying $2 \leq j \leq \frac{k}{2}$, let $\Xi=k+1-2j$, and let $S$ be a symmetrical subset of $\{1, 2, \dotsc, k\}$ (i.e., $l \in S$ if and only if $k + 1 - l \in S$) such that $1 \in S$ and $j \not \in S$. Suppose $S$ contains an integer $t$ such that $t \equiv j(\operatorname{mod} \Xi)$. Then $f\colon\F_2^k \to \F_2$ given by $$f(x)=x_j \oplus (x_{k+1-j} \oplus 1) \prod_{l \in S} x_l$$ is a $(k, n)$-lifting for any $n \geq k$.
\end{proposition}

\begin{proof}
All indices are considered $\operatorname{mod} n$. We prove the equivalent statement that $F\colon\F_2^n\to\F_2^n$ given by $$F(x)_i = x_i \oplus (x_{i+\Xi} \oplus 1) \prod_{l \in S} x_{i-j+l}$$ is a bijection. We have
\begin{align*}
F(F(x))_i \text{ }&= x_i \oplus (x_{i+\Xi} \oplus 1) \prod_{l \in S} x_{i-j+l}\\
&\oplus \left( (x_{i+\Xi} \oplus 1) \oplus (x_{i+2\Xi} \oplus 1) \prod_{l \in S} x_{i+\Xi-j+l} \right)\\
&\qquad {}\cdot \prod_{l \in S} \left( x_{i-j+l} \oplus (x_{i+\Xi-j+l} \oplus 1) \prod_{m \in S} x_{i-2j+l+m}\right)
\end{align*}
Note that all products $\prod_{m\in S}x_{i-2j+l+m}$ contain $x_{i+\Xi}$ as a factor (for $m=k+1-l$). Eliminating all terms containing a factor of the form $x_{i+s} (x_{i+s} \oplus 1)$ yields
\begin{align*}
F(F(x))_i \text{ }&= x_i \oplus (x_{i+\Xi} \oplus 1) \prod_{l \in S} x_{i-j+l} \oplus (x_{i+\Xi} \oplus 1) \prod_{l \in S} x_{i-j+l} \\
&\quad {}\oplus (x_{i+2\Xi} \oplus 1) \prod_{l \in S} x_{i-j+l} x_{i+\Xi-j+l}\\
&= x_i \oplus (x_{i+2\Xi} \oplus 1) \prod_{l \in S} x_{i-j+l} x_{i+\Xi-j+l}
\end{align*}
The product $\prod_{l \in S} x_{i-j+l} x_{i+\Xi-j+l}$ can be written as $\prod_{l \in S'} x_{i-j+l}$ where $S'$ is symmetrical ($l \in S'$ if and only if $k'+1-l \in S'$ for $k'=k+\Xi$). We have not used the assumptions that $1 \in S$, $t \in S$ or $j \not \in S$ to derive this expression, and therefore, we can iterate it without having to check whether they are satisfied for $S'$. Thus, $r$ iterations of applying the last new function to itself gives us $$F^{(2^r)}(x)_i = x_i \oplus (x_{i+2^r \Xi} \oplus 1) \prod_{l \in S, \text{ }0 \leq m \leq 2^r-1} x_{i-j+l+m\Xi}$$ Now we can apply the assumptions. Since $1 \in S$ and $j \not \in S$, the diameter of $f$ is $k$. We can assume without loss of generality that $j \leq \frac{k}{2}$ and $t>j$, and set $r=\lceil \operatorname{log}_2 (\frac{t-j}{\Xi}+1)\rceil$. For $l=t$ and $m=2^r-\frac{t-j}{\Xi} \geq 1$ we then get the factor $x_{i-j+l+m\Xi} = x_{i+2^r \Xi}$ in the product on the right, and obtain $F^{(2^r)}(x) = x$.
\end{proof}

\begin{example}
For $j=2$, $k=4$, $S=\{1,4\}$, we get $f(x)=x_2\oplus x_1(x_3\oplus 1)x_4$, that is, another way to see this function is a proper lifting.
\end{example}

\begin{proposition}
If $r$ is an integer $\geq 2$, then $f\colon\F_2^{2r} \to \F_2$ given by $$f(x)=x_r \oplus \sum_{j=1}^{r-1} (x_j \oplus 1) (x_{r+j+1} \oplus 1) \left( \prod_{k=1}^j x_{r+k} \right) \left( \prod_{k=j+1}^r (x_k \oplus 1) \oplus \prod_{k=j+1}^r x_k \right)$$ is a $(2r, n)$-lifting for all $n \geq 2r$.
\end{proposition}

\begin{proof}
All indices are considered $\operatorname{mod} n$. We prove the equivalent statement that $F\colon\F_2^n\to\F_2^n$ given by $$F(x)_i = x_i \oplus \sum_{j=1}^{r-1} (x_{i-r+j} \oplus 1) (x_{i+j+1} \oplus 1) \left( \prod_{k=1}^j x_{i+k} \right) \left( \prod_{k=j+1}^r (x_{i-r+k} \oplus 1) \oplus \prod_{k=j+1}^r x_{i-r+k} \right)$$ is a bijection. Note that the factor $$\prod_{k=j+1}^r (x_{i-r+k} \oplus 1) \oplus \prod_{k=j+1}^r x_{i-r+k}$$ is 1 if $x_i = x_{i-1} = \dotsb = x_{i-r+j+1}$ and 0 otherwise.

In order that $F(x)_i \neq x_i$, there must be at least one term of the sum that is 1, i.e., there must exist $j_0$, $1 \leq j_0 \leq r-1$ such that $x_{i-r+j_0} = x_{i+j_0+1} = 0$, $x_{i+k} = 1$ for $1 \leq k \leq j_0$, and $x_i = x_{i-1} = \dotsb = x_{i - r + j_0 + 1}$. Thus, $(x_{i-r+j_0}, \dotsc, x_{i+j_0+1})$ must consist of $l$ zeros ($1 \leq l \leq r$), followed by $r+1-l$ ones, followed by a single zero. For $-r+j_0 \leq m \leq j_0 + 1$, consider $$F(x)_{i+m} = x_{i+m} \oplus \sum_{j=1}^{r-1} A_j B_j C_j D_j$$ where
\[
\begin{gathered}
 A_j=x_{i+m-r+j} \oplus 1,\qquad
 B_j=x_{i+m+j+1} \oplus 1,\qquad
 C_j=\prod_{k=1}^j x_{i+m+k},\\
 D_j=\prod_{k=j+1}^r (x_{i+m-r+k} \oplus 1)
 \oplus \prod_{k=j+1}^r x_{i+m-r+k}.
\end{gathered}
\]

For $m=-r+j_0$, $F(x)_{i+m}=x_{i+m}=0$ because either $l=1$ and $B_j=0$ for all $j$, or $l > 1$ and $x_{i+m+1}=0 \implies C_j=0$ for all $j$. For $m=j_0$, $F(x)_{i+m}=x_{i+m}=1$ because $x_{i+m+1}=0 \implies C_j=0$ for all $j$, and for $m=j_0+1$, $F(x)_{i+m}=x_{i+m}=0$ because either $j=r-1$ and $x_{(i+m)-r+j}=x_{i+j_0}=1 \implies A_j=0$, or $j<r-1$ and $D_j=0$. In short, the first and the last two bits of the interval do not change.

For the remaining bits, first consider $l=1$. For $j=j_0-m$ we have $A_j=x_{i-r+j_0} \oplus 1=1$, $B_j=x_{i+j_0+1} \oplus 1=1$, $C_j=\prod_{k=m+1}^{j_0} x_{i+k}=1$ and $i+m-r+(j+1) = i - r + j_0 + 1$, $i+m-r+r=i+m \leq i + j_0 - 1 \implies D_j=1$. If $j < j_0 - m$, then $D_j=0$. If $j > j_0 - m$, then $A_j = 0$. Therefore, $F(x)_{i+m} = 0 \neq x_{i+m}$ for all $m$, $-r+j_0+1 \leq m \leq j_0-1$.

Suppose $l>1$ henceforth. We have $j_0 = r - l + 1$. If $m < 0$, then $x_{(i+m)+1}=0 \implies C_j=0$ for all $j$, and then $F(x)_{i+m} = x_{i+m} = 0$. If $m=0$, then $A_j B_j C_j D_j=1$ for $j=j_0$. If $j<j_0$, then $B_j=0$, and if $j>j_0$, then $C_j = 0$ since $x_{i+m+j_0+1}=0$. Thus, $F(x)_{i+m} = 1 \neq x_{i+m}$. Finally, if $m > 0$, then either $j>r-m$ and $x_{i+m-r+j}=1 \implies A_j=0$, or $j \leq r-m$ and $D_j=0$, which implies $F(x)_{i+m}=x_{i+m}=1$. Thus, iterating $x \to F(x)$ yields the cycle $(0, \dotsc, 0, 1, 0) \to (0, \dotsc, 0, 1, 1, 0) \to \dotsb \to (0, 1, \dotsc, 1, 0) \to (0, \dotsc, 0, 1, 0)$ of length $r$, and we conclude that $F^{(r)}(x) = x$.
\end{proof}

\begin{example}
For $r=2$, we get $f(x)=x_2\oplus(x_1\oplus 1)x_3(x_4\oplus 1)$, which is the Patt function, elementary equivalent to the function considered in the previous examples, yet again shown to be a proper lifting.
\end{example}

The next result deals with the existence of permutive proper liftings. We have not found any functions of this type, other than in the trivial case $k=1$.

\begin{proposition}\label{permutive-proper}
Let $g$ be a Boolean function in $k-1$ variables that depends on $x_1$ such that $g(0)=0$, and suppose $g$ satisfies at least one of the two following conditions.
\begin{enumerate}
\item For every $x \in \F_2^{k-1}$ such that $x_{k-1}=1$, there exists $l \in \{1, \dotsc, k - 1\}$ such that $g\left(x_l, x_{l+1}, \dotsc, x_{k-1}, 0, \dotsc, 0\right) = 1$.
\item For every $x \in \F_2^{k-1}$ such that $x_{k-1}=0$, there exists $l \in \{1, \dotsc, k - 1\}$ such that $g\left(x_l, x_{l+1}, \dotsc, x_{k-1}, 1, \dotsc, 1\right) = 1$.
\end{enumerate}
Then $f\colon\F_2^k \to \F_2$ given by $f(x) = g\left(x_1, \dotsc, x_{k-1}\right) \oplus x_k$ is not a proper lifting.
\end{proposition}

\begin{proof}
If $g$ satisfies the first condition, let $s_1, \dotsc, s_{k-2}$ be any $k-2$ bits and let $s_{k-1}=1$. Define $s_k, s_{k+1}, \dotsc$ recursively by $s_{m+k} = g\left(s_{m+1}, \dotsc, s_{m+k-1}\right)$ such that $f\left(s_{m+1}, \dotsc, s_{m+k}\right)$ is always equal to 0. For some $n_1$, $n_2$ with $0 \leq n_1 < n_2 \leq 2^k$ we must have $s_{n_1+i}=s_{n_2+i}$ for $1 \leq i \leq k$, and the condition ensures that those $k$ bits are not all equal to 0. With $n=n_2 - n_1$, we then have $F\left(s_{n_1}, \dotsc, s_{n_2 - 1}\right) = (0, \dotsc, 0) = F(0, \dotsc, 0)$.

Similarly, if $g$ satisfies the second condition, let $s_{k-1}=0$, and define $s_k$ etc. recursively such that $f\left(s_{m+1}, \dotsc, s_{m+k}\right)$ is always equal to 1. With $n_1$ and $n_2$ as above, the condition ensures that the $k$ bits are not all equal to 1. We can assume that $f(1, \dotsc, 1) = 1$, otherwise $f$ cannot be a proper lifting. Thus, $F\left(s_{n_1}, \dotsc, s_{n_2 - 1}\right) = (1, \dotsc, 1) = F(1, \dotsc, 1)$.
\end{proof}

If there exists a nontrivial permutive proper lifting $f$, then the associated $g$ satisfies neither (1) nor (2). However, under this condition, for each $k \in \{4, 5, 6\}$, we have verified that there exists $n \leq 2k-1$ such that $f$ is not a $(k,n)$-lifting, and thus not a proper lifting.

\section{Discussion and further work}
\label{sec:discussion}

We introduced families of proper liftings built from primitive landscapes and composition, and used them to classify proper liftings of small diameter. In particular, by comparison with Jeandel's independent enumeration of surjective rules, we proved that every non-affine proper lifting of diameter~$6$ is (up to elementary equivalence) a composition of primitive landscapes.

These permutations are shift-invariant and extend uniformly to all window sizes, but the resulting structure tends to limit linear resistance. In our computations, many rules therefore have modest nonlinearity. Nevertheless, we identify a particular degree~$5$ rule in diameter~$6$ that simultaneously minimizes our measured differential indicator and attains the largest $\operatorname{NL}$ observed. For comparison only, its $\operatorname{NL}$ matches that of Keccak’s $\chi_n$ at the corresponding window size.

\medskip

Some open problems:
\begin{itemize}
  \item[(i)] For $k=6$ all proper liftings arise as compositions of primitive landscapes; does the same hold for larger diameters, and if not, what is the counterexample with smallest diameter?

  \item[(ii)] Prove structural bounds on classical cryptographic indicators for proper liftings of fixed diameter, e.g., upper bounds on $\operatorname{NL}(F)$ and lower bounds on $\Delta(f)$ or $\operatorname{DU}(F)$, and compare them with the best examples found computationally.

  \item[(iii)] Do proper liftings of algebraic degree~$2$ exist, or more generally, what is the minimal possible degree as a function of diameter?
\end{itemize}


\appendix
\raggedbottom

\section{List of new liftings with the landscape notation}\label{appendix:landscapes}

The functions below are described as compositions of landscape functions.
The differential uniformity is given for $n \in \{6, \dotsc, 12\}$, and the nonlinearity is given for $n \in \{6, 7, 8\}$.
It seems reasonable to anticipate that $\operatorname{DU}(F)$ doubles with $n$ when $n$ gets bigger, i.e., that $2^{-n}\operatorname{DU}(F)$ stabilizes when $n$ grows.

There is only one function of algebraic degree~$3$ (which is elementary equivalent to the one discussed in \cite[Example~6.7]{JDA-thesis}):

\begingroup\small
\setlength{\arraycolsep}{3pt}
\renewcommand{\arraystretch}{0.95}
\[
\begin{array}{|c|c|c|} \hline
\text{Function} & \text{Differential uniformity} & \text{Nonlinearity}\\ \hline
(0-\star100)\circ(0-\star110)& 24, 56, 112, 216, 480, 864, 1728 & 8, 16, 32\\ \hline
\end{array}
\]
\endgroup

\newpage

Functions of algebraic degree~$4$ (the two functions with lowest $\operatorname{DU}$ values are highlighted):

\begingroup\small
\setlength{\arraycolsep}{3pt}
\renewcommand{\arraystretch}{0.95}
\[
\begin{array}{|c|c|c|} \hline
\text{Function} & \text{Differential uniformity} & \text{Nonlinearity} \\ \hline
(0\star0111)\circ(0\star1010)& 34, 72, 132, 262, 532, 1064, 2128 & 4, 8, 16 \\
(0\star1110)\circ(0\star0011)& 34, 68, 132, 258, 528, 1056, 2112 & 4, 8, 16 \\
(0\star-110)& 36, 72, 148, 304, 612, 1204, 2408 & 4, 8, 16 \\
(0\star1-10)& 34, 74, 148, 304, 612, 1204, 2408 & 4, 8, 16 \\
(0\star10)\circ(0\star1-10)& 22, 38, 80, 150, 300, 600, 1200 & 8, 16, 32 \\
(0\star-110)\circ(0\star110)& 22, 48, 104, 208, 412, 810, 1620 & 8, 16, 32 \\
(00\star101)\circ(01\star000)& 40, 70, 160, 270, 528, 1056, 2112 & 4, 8, 16 \\
(01\star100)\circ(00\star011)& 30, 70, 128, 258, 518, 1032, 2064 & 4, 8, 16 \\
(0-\star100)& 38, 76, 152, 314, 616, 1204, 2408 & 4, 8, 16 \\
(00\star100)\circ(01\star110)& 36, 64, 130, 258, 502, 980, 1960 & 4, 8, 16 \\
(0-\star110)& 36, 72, 148, 304, 612, 1204, 2408 & 4, 8, 16 \\
(01\star1-0)& 34, 74, 148, 304, 612, 1204, 2408 & 4, 8, 16 \\
(01\star1-0)\circ(01\star00)& 24, 46, 108, 216, 422, 824, 1648 & 8, 16, 32 \\
(01\star1-0)\circ(01\star01)& 26, 46, 106, 196, 402, 790, 1580 & 8, 16, 32 \\
(01\star1-0)\circ(01\star0)& 20, 36, 80, 152, 312, 610, 1220 & 12, 24, 48 \\
(1\star001)\circ(1-\star001)& 32, 62, 130, 256, 502, 1004, 2008 & 4, 8, 16 \\
(01\star110)\circ(10\star001)& 64, 128, 256, 512, 1024, 2048, 4096 & 4, 8, 16 \\
(0\star011)\circ(100\star11)\circ(10\star11) \dotsb & 34, 64, 132, 272, 544, 1088, 2176 & 0, 0, 0 \\
\dotsb \circ(0\star0011)\circ(10\star11)\circ(0\star011) & & \\
(01\star0)\circ(01\star001)\circ(10\star011)& 20, 40, 78, 174, 338, 682, 1364 & 8, 16, 32 \\
(0\star10)\circ(10\star0-1)\circ(10\star10)& 20, 44, 80, 176, 360, 714, 1428 & 8, 16, 32 \\
(0\star110)\circ(01\star110)\circ(10\star0-1)\circ(10\star110)& 34, 50, 80, 174, 358, 716, 1432 & 8, 16, 32 \\
(0\star10)\circ(10\star0-1)& 20, 38, 78, 168, 332, 664, 1328 & 8, 24, 48 \\
(00\star10)\circ(0\star110)& 22, 52, 98, 190, 380, 760, 1520 & 8, 16, 32 \\
(0\star110)\circ(01\star110)\circ(00\star011)& 22, 44, 88, 178, 356, 708, 1416 & 8, 16, 32 \\
(0\star110)\circ(00\star10)& 22, 52, 98, 190, 380, 760, 1520 & 8, 16, 32 \\
(0\star10)\circ(0\star110)\circ(0\star10)& 34, 64, 134, 266, 518, 1036, 2072 & 0, 0, 0 \\
(0\star110)\circ(01\star110)\circ(10\star001)& 22, 44, 100, 176, 350, 696, 1392 & 8, 16, 32 \\
(0\star110)\circ(01\star110)\circ(10\star011)& 22, 46, 90, 180, 368, 738, 1476 & 8, 16, 32 \\
(0\star110)\circ(10\star10)& 26, 46, 102, 192, 384, 768, 1536 & 8, 16, 32 \\
\cellcolor{lightgray}(0\star10)\circ(0\star110)& \cellcolor{lightgray}18, 36, 68, 132, 264, 528, 1056 & 12, 24, 48 \\
\cellcolor{lightgray} (00\star10)\circ(0\star110)\circ(0\star10)& \cellcolor{lightgray} 18, 32, 66, 120, 234, 468, 936 & 0, 24, 48 \\
(10\star111)\circ(00\star100)& 40, 60, 134, 268, 512, 1026, 2052 & 4, 8, 16 \\
(10\star111)\circ(00\star101)& 40, 64, 160, 244, 502, 1008, 2016 & 4, 8, 16 \\
(01\star000)\circ(10\star111)& 64, 128, 256, 512, 1024, 2048, 4096 & 4, 8, 16 \\
(0\star10)\circ(10\star111)\circ(00\star101)& 22, 46, 102, 202, 396, 802, 1604 & 8, 16, 32 \\
(0\star10)\circ(10\star111)\circ(00\star100)& 22, 52, 90, 198, 396, 796, 1592 & 8, 16, 32 \\
(00\star011)\circ(11\star001)& 34, 72, 130, 262, 518, 1036, 2072 & 4, 8, 16 \\
(00\star101)\circ(11\star010)& 64, 128, 256, 512, 1024, 2048, 4096 & 4, 8, 16 \\
(0\star011)\circ(1-\star011)& 32, 62, 130, 260, 506, 1012, 2024 & 4, 8, 16 \\
(0\star011)\circ(11\star011)\circ(11\star001)& 24, 46, 90, 192, 378, 748, 1496 & 8, 16, 32 \\
(0\star011)\circ(11\star01)& 24, 52, 106, 200, 400, 800, 1600 & 8, 16, 32 \\
(001\star11)\circ(110\star00)& 64, 128, 256, 512, 1024, 2048, 4096 & 4, 8, 16 \\ \hline
\end{array}
\]
\endgroup

\newpage
Functions of algebraic degree~$5$ (the two functions with lowest $\operatorname{DU}$ values are highlighted):

\begingroup\small
\setlength{\arraycolsep}{3pt}
\renewcommand{\arraystretch}{0.95}
\[
\begin{array}{|c|c|c|} \hline
\text{Function} & \text{Differential uniformity} & \text{Nonlinearity} \\ \hline
(0\star0011)& 46, 90, 180, 364, 732, 1472, 2944 & 2, 4, 8 \\
(0\star0110)& 46, 88, 180, 360, 724, 1440, 2880 & 2, 4, 8 \\
(0\star0111)& 46, 90, 184, 368, 740, 1488, 2976 & 2, 4, 8 \\
(0\star0011)\circ(0\star011)& 28, 52, 114, 244, 502, 1008, 2016 & 6, 12, 24 \\
(0\star1010)& 44, 92, 184, 368, 736, 1468, 2936 & 2, 4, 8 \\
(0\star10)\circ(0\star1010)& 34, 62, 116, 232, 460, 928, 1856 & 4, 12, 24 \\
(0\star10)\circ(0\star011)& 12, 24, 58, 118, 236, 472, 944 & 12, 24, 48 \\
(0\star110)\circ(0\star0011)& 32, 58, 110, 220, 442, 884, 1768 & 6, 12, 24 \\
(0\star0110)\circ(0\star110)& 28, 50, 106, 208, 408, 814, 1628 & 6, 12, 24 \\
(0\star0110)\circ(0\star110)\circ(0\star0011)& 24, 44, 82, 166, 330, 660, 1320 & 8, 16, 32 \\
(0\star110)\circ(0\star0111)& 32, 50, 108, 216, 436, 874, 1748 & 6, 12, 24 \\
(0\star1110)& 46, 90, 180, 364, 730, 1454, 2908 & 2, 4, 8 \\
(0\star10)\circ(0\star1110)& 26, 50, 106, 202, 408, 810, 1620 & 10, 20, 40 \\
(0\star1110)\circ(0\star110)& 26, 54, 112, 224, 444, 884, 1768 & 6, 12, 24 \\
(0\star1110)\circ(0\star110)\circ(0\star0111)& 22, 48, 90, 186, 370, 738, 1476 & 8, 16, 32 \\
(0\star1110)\circ(0\star110)\circ(0\star10)& 16, 34, 68, 132, 262, 520, 1040 & 14, 28, 56 \\
(00\star011)& 46, 90, 182, 368, 736, 1480, 2960 & 2, 4, 8 \\
(00\star100)& 46, 94, 198, 386, 750, 1492, 2984 & 2, 4, 8 \\
(00\star101)& 44, 88, 178, 360, 728, 1464, 2928 & 2, 4, 8 \\
(00\star110)& 46, 90, 180, 368, 732, 1460, 2920 & 2, 4, 8 \\
(00\star110)\circ(00\star10)& 28, 60, 120, 236, 472, 940, 1880 & 6, 12, 24 \\
(01\star000)& 46, 92, 186, 376, 748, 1492, 2984 & 2, 4, 8 \\
(01\star001)& 44, 92, 198, 364, 728, 1464, 2928 & 2, 4, 8 \\
(01\star0)\circ(01\star001)& 34, 66, 126, 254, 508, 1012, 2024 & 4, 12, 24 \\
(01\star0)\circ(01\star000)& 34, 66, 116, 236, 468, 944, 1888 & 4, 12, 24 \\
(01\star0)\circ(00\star011)& 30, 48, 92, 182, 364, 736, 1472 & 4, 20, 40 \\
(01\star100)& 46, 90, 178, 364, 728, 1448, 2896 & 2, 4, 8 \\
(0-\star100)\circ(00\star110)& 34, 60, 118, 264, 500, 964, 1928 & 6, 12, 24 \\
(01\star100)\circ(01\star0)& 22, 40, 84, 162, 324, 646, 1292 & 10, 20, 40 \\
(01\star110)& 46, 90, 182, 360, 726, 1446, 2892 & 2, 4, 8 \\
(01\star110)\circ(01\star0)& 18, 34, 72, 136, 276, 548, 1096 & 10, 20, 40 \\
(01\star110)\circ(01\star0)\circ(00\star011)& 20, 36, 72, 126, 260, 518, 1036 & 8, 24, 48 \\
(00\star100)\circ(01\star1-0)& 30, 64, 128, 256, 522, 1006, 2012 & 6, 12, 24 \\
(01\star1-0)\circ(01\star000)& 30, 56, 120, 264, 520, 1002, 2004 & 6, 12, 24 \\
(01\star1-0)\circ(01\star001)& 32, 52, 108, 228, 440, 860, 1720 & 6, 12, 24 \\
(01\star1-0)\circ(01\star001)\circ(01\star0)& 26, 46, 92, 196, 388, 752, 1504 & 10, 20, 40 \\
(01\star1-0)\circ(01\star000)\circ(01\star0)& 22, 38, 82, 156, 304, 604, 1208 & 10, 20, 40 \\
(0\star011)\circ(100\star11)\circ(10\star11)\circ(0\star0011) \dotsb& 18, 42, 94, 196, 374, 744, 1488 & 10, 20, 32 \\
\dotsb \circ(10\star11)\circ(0\star011)\circ(00\star100) & & \\
(0\star011)\circ(100\star11)\circ(10\star11) \dotsb& 22, 48, 102, 220, 440, 878, 1756 & 8, 16, 32 \\
\dotsb \circ(0\star0011)\circ(10\star11)\circ(0\star0111) & & \\
(01\star110)\circ(0\star011)& 24, 58, 108, 224, 458, 912, 1824 & 6, 12, 24 \\
(0\star011)\circ(100\star11)\circ(10\star11)\circ(0\star0011) \dotsb& 22, 50, 92, 190, 388, 770, 1540 & 10, 20, 40 \\
\dotsb \circ(10\star11)\circ(0\star011)\circ(10\star001) & & \\
(0\star10)\circ(00\star101)& 34, 54, 122, 238, 468, 944, 1888 & 4, 12, 24 \\
(0\star10)\circ(00\star100)& 34, 58, 110, 232, 470, 932, 1864 & 4, 12, 24 \\ \hline
\end{array}
\]
\endgroup

\begingroup\small
\setlength{\arraycolsep}{3pt}
\renewcommand{\arraystretch}{0.95}
\[
\begin{array}{|c|c|c|} \hline
\text{Function} & \text{Differential uniformity} & \text{Nonlinearity}\\ \hline
(10\star10)\circ(00\star110)& 34, 54, 116, 218, 436, 868, 1736 & 4, 12, 24 \\
(00\star110)\circ(0\star10)& 22, 42, 86, 160, 328, 652, 1304 & 10, 20, 40 \\
(0\star10)\circ(10\star001)& 30, 42, 84, 176, 352, 704, 1408 & 4, 20, 40 \\
(0\star10)\circ(10\star011)& 20, 32, 76, 156, 310, 620, 1240 & 10, 20, 40 \\
(00\star110)\circ(00\star10)\circ(0\star110)& 22, 44, 100, 180, 358, 716, 1432 & 8, 16, 32 \\
(0\star110)\circ(01\star110)& 28, 62, 126, 252, 498, 990, 1980 & 6, 12, 24 \\
(10\star110)\circ(00\star10)& 28, 54, 118, 212, 422, 848, 1696 & 6, 12, 24 \\
(10\star110)\circ(00\star10)\circ(0\star110)& 24, 46, 102, 196, 380, 764, 1528 & 8, 16, 32 \\
(0\star110)\circ(10\star001)& 32, 50, 102, 192, 400, 794, 1588 & 6, 12, 24 \\
(0\star110)\circ(0\star011)& 24, 38, 84, 168, 336, 672, 1344 & 8, 16, 32 \\
(00\star110)\circ(00\star10)\circ(0\star110)\circ(10\star011)& 22, 28, 64, 126, 256, 510, 1020 & 10, 20, 40 \\
(0\star110)\circ(01\star110)\circ(0\star011)& 22, 40, 80, 166, 330, 656, 1312 & 10, 20, 40 \\
(0\star10)\circ(0\star110)\circ(0\star10)\circ(10\star011)& 34, 38, 76, 158, 328, 652, 1304 & 8, 20, 40 \\
(0\star110)\circ(01\star110)\circ(10\star0-1)& 24, 36, 72, 146, 300, 594, 1188 & 6, 20, 40 \\
\cellcolor{lightgray} (0\star10)\circ(0\star110)\circ(01\star00)& \cellcolor{lightgray} 10, 26, 42, 72, 144, 288, 576 & 12, 32, 64 \\
(0\star10)\circ(0\star110)\circ(000\star10)& 16, 34, 54, 110, 214, 428, 856 & 12, 24, 48 \\
(10\star110)\circ(0\star10)& 18, 48, 80, 162, 316, 636, 1272 & 10, 20, 40 \\
(0\star10)\circ(0\star110)\circ(100\star10)& 16, 34, 58, 110, 214, 432, 864 & 12, 24, 48 \\
(0\star10)\circ(0\star110)\circ(10\star001)& 16, 36, 56, 106, 220, 440, 880 & 12, 28, 56 \\
(10\star110)\circ(0\star10)\circ(10\star001)& 20, 34, 60, 128, 250, 502, 1004 & 8, 24, 48 \\
\cellcolor{lightgray}(0\star110)\circ(0\star10)\circ(10\star011)& \cellcolor{lightgray}16, 24, 46, 96, 194, 388, 776 & 12, 28, 56 \\
(00\star10)\circ(10\star111)& 28, 50, 122, 200, 404, 806, 1612 & 6, 12, 24 \\
(10\star111)\circ(00\star101)\circ(01\star000)& 40, 72, 144, 224, 464, 728, 1600 & 6, 12, 24 \\
(0\star10)\circ(10\star0-1)\circ(10\star110)\circ(10\star1)& 16, 28, 72, 144, 294, 574, 1148 & 10, 20, 40 \\
(0\star10)\circ(10\star111)& 20, 48, 90, 194, 386, 770, 1540 & 10, 20, 40 \\
(10\star1)\circ(10\star110)\circ(0-\star100)& 18, 40, 72, 162, 314, 644, 1288 & 10, 20, 40 \\
(0\star10)\circ(10\star0-1)\circ(10\star111)& 20, 38, 60, 170, 324, 644, 1288 & 12, 28, 56 \\
(10\star1)\circ(0\star10)& 28, 38, 100, 152, 304, 608, 1216 & 8, 16, 32 \\
(10\star1)\circ(00\star10)& 20, 38, 76, 154, 308, 616, 1232 & 12, 24, 48 \\
(00\star101)\circ(01\star000)\circ(11\star010)& 40, 72, 144, 224, 504, 728, 1744 & 6, 12, 24 \\
(0\star011)\circ(11\star010)& 32, 50, 116, 206, 408, 820, 1640 & 6, 12, 24 \\
(0\star011)\circ(11\star011)& 30, 64, 128, 254, 504, 1004, 2008 & 6, 12, 24 \\
(001\star1)\circ(001\star00)\circ(011\star00)\circ(01\star00)& 24, 32, 72, 138, 274, 548, 1096 & 8, 16, 32 \\
(001\star1)\circ(110\star00)& 46, 72, 128, 260, 504, 860, 1672 & 6, 12, 24 \\ \hline
\end{array}
\]
\endgroup

	\bibliographystyle{plain}

\end{document}